\documentclass[a4paper,12pt]{article}
\usepackage[T2A]{fontenc} 
\usepackage[english]{babel}

\usepackage{mathrsfs}
\usepackage[top=2cm,left=3cm,right=1.5cm,bottom=2cm]{geometry}
\usepackage{amsmath}
\usepackage{amsfonts}
\usepackage{amssymb}
\usepackage{amsthm} 
\usepackage{enumerate}
\usepackage{hyperref}

\numberwithin{equation}{subsection}

\theoremstyle{plain}
\newtheorem{theorem}{Theorem}[section]
\newtheorem{corollary}[theorem]{Corollary}
\newtheorem{lemma}[theorem]{Lemma}

\theoremstyle{definition}
\newtheorem{definition}[theorem]{Definition}

\theoremstyle{remark}
\newtheorem{remark}[theorem]{Remark}
\newtheorem{example}[theorem]{Example}
\numberwithin{equation}{section}

\begin{document}



\begin{center}\textbf{Correct Singular Perturbations of the Laplace Operator with the Spectrum of the Unperturbed Operator}\end{center}

\begin{center}\textbf{B.\,N. Biyarov, D.\,A. Svistunov, G.\,K. Abdrasheva}\end{center}

\\


\begin{abstract}
The work is devoted to the study of Laplace operator when the potential is a singular generalized function and plays the role of a singular perturbation of a Laplace operator. Abstract theorem obtained earlier by the authors B.N.\,Biyarov and G.K.\,Abdrasheva applies to this. The main purpose of the study is the spectral issue. Singular perturbations for differential operators have been studied by many authors for the mathematical substantiation of solvable models of quantum mechanics, atomic physics, and solid state physics. In all these cases, the problems were self-adjoint. In this paper, we consider non-self-adjoint singular perturbation problems. A new method has been developed that allows investigating the considered problems.
\end{abstract}

\section{Introduction}
\label{sec:1} 
Let us present some definitions, notation, and terminology.

In a Hilbert space $H$, we consider a linear operator $L$ with  domain $D(L)$ and range $R(L)$. By the \textit{kernel} of the operator $L$ we mean the set
\[\mbox{Ker}\,L=\bigl\{f\in D(L): \; Lf=0\bigl\}.\]

\begin{definition}
\label{def1}  
An operator $L$ is called a \textit{restriction} of an operator $L_1$, and $L_1$ is called an \textit{extension} of an operator $L$, briefly $L\subset L_1$, if:

1) $D(L)\subset D(L_1)$,

2) $Lf=L_1f$ for all $f$ from $D(L)$.
\end{definition}

\begin{definition}
\label{def2}
A linear closed operator $L_0$ in a Hilbert space $H$ is called \textit{minimal} if  there exists a bounded inverse operator $L_0^{-1}$ on $R(L_0)$ and $R(L_0) \not=H$.
\end{definition}

\begin{definition}
\label{def3}
A linear closed operator $\widehat{L}$ in a Hilbert space $H$ is called \textit{maximal} if $R(\widehat{L})=H$ and $\mbox{Ker}\, \widehat{L} \not=\{0\}$. 
\end{definition}

\begin{definition}
\label{def4}
A linear closed operator $L$ in a Hilbert space $H$ is called \textit{correct} if there exists a bounded inverse operator $L^{-1}$ defined on all of $H$. 
\end{definition}

\begin{definition}
\label{def5}
We say that a correct operator $L$ in a Hilbert space $H$ is a \textit{correct extension} of minimal operator $L_0$ (\textit{correct restriction} of maximal operator $\widehat{L}$) if $L_0\subset L$ ($L\subset \widehat{L}$).
\end{definition}

\begin{definition}
\label{def6}
We say that a correct operator $L$ in a Hilbert space $H$ is a \textit{boundary correct} extension of a minimal operator $L_0$ with respect to a maximal operator $\widehat{L}$ if $L$ is simultaneously a correct restriction of the maximal operator $\widehat{L}$ and a correct extension of the minimal operator $L_0$, that is, $L_0\subset L \subset \widehat{L}$.
\end{definition}

Let $\widehat{L}$ be a maximal linear operator in a Hilbert space $H$, let $L$ be any known correct restriction of $\widehat{L}$, and let $K$ be an arbitrary linear bounded (in $H$) operator satisfying the following condition:
\[
R(K)\subset \mbox{Ker}\, \widehat{L}.
\]
Then the operator $L_K^{-1}$ defined by the formula (see {\cite{Kokebaev}})
\begin{equation}\label{eq:1.1}
L_K^{-1}f=L^{-1}f+Kf,
\end{equation}
describes the inverse operators to all possible correct restrictions $L_K$ of $\widehat{L}$, i.e., $L_K\subset \widehat{L}$.

Let $L_0$ be a minimal operator in a Hilbert space $H$, let $L$ be any known correct extension of $L_0$, and let $K$ be a linear bounded operator in $H$ satisfying the conditions

a) $R(L_0)\subset \mbox{Ker}\,K$,

b) $\mbox{Ker}\,(L^{-1}+K)=\{0\}$,
\newline
then the operator $L_K^{-1}$ defined by formula \eqref{eq:1.1}
describes the inverse operators to all possible correct extensions $L_K$ of  $L_0$ (see {\cite{Kokebaev}}).

Let $L$ be any known boundary correct extension of $L_0$, i.e., $L_0\subset L\subset \widehat{L}$. The existence of at least one boundary correct extension $L$ was proved by Vishik in {\cite{Vishik}}. Let $K$ be a linear bounded (in $H$) operator satisfying the conditions

a) $R(L_0)\subset \mbox{Ker}\,K$,

b) $R(K)\subset \mbox{Ker}\,\widehat{L}$,
\newline
then the operator $L_K^{-1}$ defined by formula \eqref{eq:1.1}
describes the inverse operators to all possible boundary correct extensions $L_K$ of $L_0$ (see {\cite{Kokebaev}}). 

\begin{definition}
\label{def7} A bounded operator $A$ in a Hilbert space $H$ is called \textit{quasinilpotent} if its spectral
radius is zero, that is, the spectrum consists of the single point zero.
\end{definition}

\begin{definition}
\label{def8} An operator $A$ in a Hilbert space $H$ is called a \textit{Volterra operator} if $A$ is compact and quasinilpotent.
\end{definition}

\begin{definition}
\label{def9} A correct restriction $L$ of a maximal operator $\widehat{L} \;(L\subset \widehat{L})$, a correct extension $L$ of a minimal operator $L_0 \; (L_0 \subset L)$ or a boundary correct extension $L$ of a minimal operator $L_0$ with respect to a maximal operator $\widehat{L} \; (L_0\subset L \subset \widehat{L})$, will be called \textit{Volterra} if the inverse operator $L^{-1}$ is a
Volterra operator.
\end{definition}

\begin{definition} \label{def10} 
A densely defined closed linear operator $A$ in a Hilbert space $H$ is called \textit{formally normal} if
\[ D(A)\subset D(A^*), \quad \|Af\|=\|A^*f\| \quad \mbox{for all} \; f \in D(A). \]
\end{definition}

\begin{definition} \label{def11} 
A formally normal operator $A$ is called \textit{normal} if
\[D(A)=D(A^*).\]
\end{definition}

\section{Preliminaries}
\label{sec:2}
In this section, we present some results for correct restrictions and extensions {\cite{Biyarov3}} which are used in Section 3.

Let $L_0$ be some minimal operator, and let $M_0$ be another minimal operator related to $L_0$ by the equation $(L_0u, v)=(u, M_0v)$ for all $u\in D(L_0)$ and $v\in D(M_0)$. Then $\widehat{L}=M_0^*$ and $\widehat{M}=L_0^*$ are maximal operators such that $L_0\subset \widehat{L}$ and $M_0\subset \widehat{M}$. The existence of at least one boundary correct extension $L$ was proved by Vishik in {\cite{Vishik}}, that is, $L_0\subset L\subset \widehat{L}$.
In this case, $L^*$ is a boundary correct extension of the minimal operator $M_0$, that is, $M_0\subset L^*\subset \widehat{M}$. 
The inverse operators to all possible correct restrictions $L_K$ of the maximal operator $\widehat{L}$ have the form \eqref{eq:1.1}, then $D(L_K)$ is dense in $H$ if and only if $\mbox{Ker}\, (I+K^*L^*)=\{0\}$.
Thus, it is obvious that any correct extension $M_K$ of $M_0$ is the adjoint of some correct restriction $L_K$ with dense domain, and vice versa {\cite{Biyarov1}}. Finally, all possible correct extensions $M_K$ of $M_0$ have inverses of the form

\begin{equation}\label{eq:2.1}
M_K^{-1}f=(L_K^*)^{-1}f=(L^*)^{-1}f+K^*f,
\end{equation}
where $K$ is an arbitrary bounded linear operator in $H$ with $R(K)\subset \mbox{Ker}\, \widehat{L}$ such that $\mbox{Ker}\, (I+K^*L^*)=\{0\}$. It is also clear that $R(M_0)\subset \mbox{Ker}\, K^*$.
In particular, $M_K$ is a boundary correct extension of $M_0$ if and only if $R(M_0)\subset \mbox{Ker}\, K^*$ and $R(K^*)\subset \mbox{Ker}\, \widehat{M}$.

\begin{lemma}\label{lem:2.0}
Let $L$ be a densely defined correct restriction of the maximal operator $\widehat{L}$ in a Hilbert space $H$. Then the operator $KL$ is bounded on $D(L)$ (that is, $\overline{KL}$ is bounded in  $H$)  if and only if 
\[R(K^*)\subset D(L^*).\]
\end{lemma}

\begin{proof}
Let $R(K^*)\subset D(L^*)$. Then, by virtue of  $(KL)^*=L^*K^*$, we have $\overline{KL}$ is bounded in $H$, where $\overline{KL}$ is the closure of the operator $KL$ in $H$. Here we used the boundedness of the operator $L^*K^*$. Then the operator $KL$ is bounded on $D(L)$.
Conversely, let the $KL$ be bounded on $D(L)$. Then $\overline{KL}$ is bounded in $H$, by virtue of $(KL)^*=(\overline{KL})^*$ and that $(KL)^*$ is defined on the whole space $H$. Then the operator $K^*$ translates any element $f$ from $H$ to $D(L^*)$. Indeed, for any element $g$ of $D(L)$ we have
\[(Lg, K^*f)=(KLg, f)=(g, (KL)^*f).\]
Therefore, $K^*f$  belongs to the domain $D(L^*)$. The Lemma \ref{lem:2.0} is proved. 
\end{proof}

\begin{lemma}\label{lem:2.1}
Let $L_K$ be a densely defined correct restriction of the maximal operator $\widehat{L}$ in a Hilbert space $H$. Then $D(L^*)= D(L_K^*)$ if and only if $R(K^*)\subset D(L^*)\cap D(L_K^*)$, where $L$ and $K$ are the operators from the representation \eqref{eq:1.1}.
\end{lemma}

\begin{proof}
If $D(L^*)= D(L_K^*)$ then from the representation \eqref{eq:1.1}, we easily get 
\[R(K^*)\subset D(L^*)\cap D(L_K^*)=D(L^*)= D(L_K^*)\]
Let us prove the converse. If
\[R(K^*)\subset D(L^*)\cap D(L_K^*),\]
then we obtain
\begin{equation}\label{eq:2.2}
(L_K^*)^{-1}f=(L^*)^{-1}f+K^*f=(L^*)^{-1}(I+L^*K^*)f,
\end{equation}
\begin{equation}\label{eq:2.3}
(L^*)^{-1}f=(L_K^*)^{-1}f-K^*f=(L_K^*)^{-1}(I-L_K^*K^*)f,
\end{equation}
for all $f$ in $H$. 
It follows from \eqref{eq:2.2} that $D(L_K^*)\subset D(L^*)$, and from \eqref{eq:2.3} it implies that $D(L^*)\subset D(L_K^*)$. Thus $D(L^*)= D(L_K^*)$. 
The Lemma \ref{lem:2.1} is proved.
\end{proof}

\begin{corollary}\label{cor:2.3}
Let $L_K$ be densely defined correct rectriction of the maximal operator $\widehat{L}$ in a Hilbert space $H$. If $R(K^*)\subset D(L^*)$ and $\overline{KL}$ is compact operator in $H$ then $D(L^*)=D(L_K^*)$.
\end{corollary}

\begin{proof}
Compactness of $\overline{KL}$ implies compactness of $L^*K^*$. Then $R(I+L^*K^*)$  is a closed subspace in $H$. It follows from densely definiteness of $L_K$  that $R(I+L^*K^*)$ is a dense set in $H$. Hence $R(I+L^*K^*)=H$. Then from the equality \eqref{eq:2.2} we get $D(L^*)=D(L_K^*)$. 
The Corollary \ref{cor:2.3} of Lemma \ref{lem:2.1} is proved.
\end{proof}

\begin{lemma}\label{lem:2.2}
If $R(K^*)\subset D(L^*)\cap D(L_K^*)$ then a bounded operators $I+L^*K^*$ and $I-L_K^*K^*$ from \eqref{eq:2.2} and \eqref{eq:2.3}, respectively, have a bounded inverse defined on $H$.
\end{lemma}

\begin{proof}
By virtue of the density of the domains of the operators $L_K^*$ and $L^*$ imply that the operators $I+L^*K^*$ and $I-L_K^*K^*$ are invertible. Since from \eqref{eq:2.2} and \eqref{eq:2.3} we have 
$\mbox{Ker}\, (I+L^*K^*)=\{0\}$ and $\mbox{Ker}\, (I-L_K^*K^*)=\{0\}$, respectively.
From the representations \eqref{eq:2.2} and \eqref{eq:2.3} we also note that $R(I+L^*K^*)=H$ and $R(I-L_K^*K^*)=H$, since $D(L^*)= D(L_K^*)$.
The inverse operators $(I+L^*K^*)^{-1}$ and $(I-L_K^*K^*)^{-1}$ of the closed operators $I-L_K^*K^*$
 and $I+L^*K^*$, respectively, are closed. Then the closed operators $(I+L^*K^*)^{-1}$ and $(I-L_K^*K^*)^{-1}$, defined on the whole of $H$, are bounded. The Lemma \ref{lem:2.2} is proved.
\end{proof}

Under the conditions of Lemma \ref{lem:2.2} the operators $KL$ and $KL_K$ will be (see {\cite{Biyarov2}}) a part of bounded operators $\overline{KL}$ and $\overline{KL_K}$, respectively, where the bar denotes the closure of operators in $H$. Thus $(I-L_K^*K^*)^{-1}=I+L^*K^*$ and $(I-\overline{KL_K})^{-1}=I+\overline{KL}$.

Next we consider the following statement

\begin{theorem}\label{theorem:2.1}
Let $L_K$ be a densely defined correct restriction of the maximal operator $\widehat{L}$ in a Hilbert space $H$. If $R(K^*)\subset D(L^*)\cap D(L_K^*)$, where $L$ and $K$ are the operators from the representation \eqref{eq:1.1} then
\begin{enumerate}
\item{The operator $B_K=(I+\overline{KL})L_K$ is relatively bounded correct perturbations of correct restriction $L_K$ and the spectra of the operators $B_K$ and $L$ coincide, that is, $\sigma (B_K)=\sigma (L)$;}
\item{The operator $L$ is quasinilpotent (the Volterra) boundary correct extension of $L_0$, and $B_K$ is a quasinilpotent (the Volterra) correct operator simultaneously;}
\item{If $L$ is an operator with discrete spectrum then the system of root vectors of the operator $L$ is complete (the basis) in $H$ if and only if the system of root vectors of the operator $B_K$ is complete (the basis) in $H$;}
\item{In particular, when $L$ is a normal operator with discrete spectrum, then the system of root vectors of the operator $B_K$ form a Riesz basis in $H$.}
\end{enumerate}
\end{theorem}

\begin{proof}
\begin{enumerate}
\item{Note that $B_K^{-1}=L_K^{-1}(I-\overline{KL_K})$, and $(I-\overline{KL_K})L_K^{-1}=L_K^{-1}-K=L^{-1}$. The correctness of the operator $B_K$ is obvious. For bounded operators $R$ and $S$ is known (see {\cite{Barnes}}) the property $\sigma (RS)\setminus \{0\}=\sigma (SR)\setminus \{0\}$. Thus, the item 1 is proved.}
\item{Note that $B_K^{-1}=(I-\overline{KL_K})^{-1}L^{-1} (I-\overline{KL_K})$.  It follows easily from Lemma \ref{lem:2.1} and Lemma \ref{lem:2.2} that the operators $I-\overline{KL_K}$ and $(I-\overline{KL_K})^{-1}$ are bounded and defined on the whole of $H$. It is then obvious that the operators $L^{-1}$ and $B_K^{-1}$ is quasinilpotent (the Volterra) simultaneously. The item 2 is proved.} 
\item{From the known facts of functional analysis (see {\cite{Taldykin}}) imply that the system of root vectors of the operators  $L$ and  $B_K$ are complete (the basis) simultaneously.}
\item{The system of root vectors of the normal discrete correct operator $L$ form an orthonormal basis in $H$. Then the system of root vectors of the correct operator $B_K$ form a Riesz basis in $H$.}
\end{enumerate}
The Theorem \ref{theorem:2.1} is proved. 
\end{proof}

\begin{example}\label{exam:2.1}
In the Hilbert space $L_2(0, 1)$, let us consider the minimal operator $L_0$ generated by  the differentiation operator
\[\widehat{L} y=y'=f \,\,\, \mbox{for all} \,\,\, f \, \in \, L_2(0, 1).\]
Then 
\[D(L_0)=\{y\in W_2^1(0, 1): \,\, y(0)=y(1)=0\}.\]
The action of the maximum operator $\widehat{M}=L_0^*$ has the form
\[\widehat{M} v=-v'=g \quad \mbox{for all} \,\,\, g\, \in \, L_2(0, 1).\]
Then 
\[D(M_0)=\{v\in W_2^1(0, 1): \,\, v(0)=v(1)=0\}.\]
As a fixed boundary correct extension $L$ of $L_0$ we take the operator acting as the maximal operator $\widehat{L} $ on the domain
\[D(L)=\{y\in D(\widehat{L}): \, y(0)=0\}.\]
Then all possible correct restriction $L_K$ of $\widehat{L}$ have the following inverse
\[y=L_K^{-1}f=L^{-1}+Kf=\int_0^xf(t)dt+\int_0^1f(t)\overline{\sigma(t)}dt,\]
where $\sigma(x)\in L_2(0, 1)$ defines the operator $K$.
The domain $D(L_K)$ of $L_K$ is defined as
\[D(L_K)=\{y\in W_2^1(0, 1):\, y(0)=\int_0^1y'(t)\overline{\sigma(t)}dt\}.\]
Then $D(L_K)$ is not dense in $L_2(0, 1)$ if and only if $\sigma(x)\in W_2^1(0, 1)$, $\sigma(1)=0$, and $\sigma(0)=-1$.
If we exclude such $\sigma(x)$ from $L_2(0, 1)$ then there exists $L_K^*$ which have an inverse of the form
\[v=(L_K^*)^{-1}g=(L_K^{-1})^*g=(L^*)^{-1}g+K^*g \quad \mbox{for all} \,\, \,g\in L_2(0, 1).\]
This is a description of inverse operators of all possible correct extensions $L_K^*$ of $M_0$.
Let the condition of Theorem \ref{theorem:2.1} holds. Then $\sigma(x)\in W_2^1(0, 1)$, $\sigma(1)=0$, and $\sigma(0)\neq-1$. Let us construct the following operators
\[ \begin{split}
&\overline{KL}f=-\int_0^1f(t)\sigma'(t)dt, \\
&\overline{KL_K}f=-\dfrac{1}{1+\overline{\sigma(0)}}\int_0^1f(t)\sigma'(t)dt.
\end{split}\]
Note that
\[ \begin{split}
&L_K^* v=-v'(x)+\dfrac{\sigma'(x)}{1+\sigma(0)}v(0)=f(x), \\
&D(L_K^*)=D(L^*)=\{v\in W_2^1(0, 1):\, \,v(1)=0\}.
\end{split}\]
Then the operator $B_K$ has the following form
\[ \begin{split}
&B_Ku=u'(x)-\int_0^1u'(t)\overline{\sigma'(t)}dt=f(x), \\
&D(B_K)=D(L_K)=\{u\in W_2^1(0, 1):\, \,u(0)=\int_0^1u'(t)\overline{\sigma(t)} dt\},
\end{split}\]
where $\sigma(x)\in W_2^1(0, 1)$, $\sigma(1)=0$, and $\sigma(0)\neq-1$.
By virtue of Theorem \ref{theorem:2.1} $B_K$ is a Volterra correct operator.
We know that for a first order differentiation operator there are no Volterra correct restrictions or correct extensions, except the Cauchy problem at some point $x=d $, $0\leq d\leq1$.
But the operator $B_K$ is neither correct restriction of $\widehat{L}$ nor correct extension of $L_0$.
This Volterra problem obtained by the perturbation of the differentiation operator itself and the boundary conditions of Cauchy simultaneously.
\end{example}

\begin{example}\label{exam:2.2}
If in Example \ref{exam:2.1} as a fixed boundary correct operator $L$ we take the operator $\widehat{L}$ with the domain
\[D(L)=\{y\in W_2^1(0, 1): \,\, y(0)+y(1)=0\},\]
then $L$ is a normal operator. In this case, the operator $B_K$ has the form
\[ \begin{split}
&B_Ky=y'(x)-\int_0^1y'(t)\overline{\sigma'(t)}dt=f(x), \\
&D(B_K)=\{y\in W_2^1(0, 1):\, \,y(0)+y(1)=2\int_0^1y'(t)\overline{\sigma(t)} dt\},
\end{split}\]
where $\sigma(x)\in W_2^1(0, 1)$, $\sigma(0)+\sigma(1)=0$, and $\sigma(0)\neq-\frac{1}{2}$.
The operator $B_K$ is correct and the system of root vectors form a Riesz basis in $L_2(0, 1)$.
The eigenvalues of the normal operator $L$ and the correct operator  $B_K$ coincide.
\end{example}

\begin{corollary}\label{cor:2.8}
The results of Theorem \ref{theorem:2.1} are also valid for the operator $B_K^*=L_K^*(I+L^*K^*)$. All four items will take place for a pair of operators $B_K^*$ and $L^*$.
\end{corollary}

\begin{remark}\label{rem:2.9}
The results of Examples \ref{exam:2.1}--\ref{exam:2.2} are also valid for the operator $B_K^*$.
\end{remark}
\[ \begin{split}
&B_K^*v=-\dfrac{d}{dx}[v(x)-\sigma'(x)\int_0^1v(t)dt]=f, \\
&D(B_K^*)=\{v\in L_2(0, 1):\, \,v(x)-\sigma'(x)\int_0^1v(t)dt\in D(L^*)\},
\end{split}\]
where $\sigma(x)\in W_2^1(0, 1)$, $\sigma(1)=0$, and $\sigma(0)\neq-1$, in the case of Example \ref{exam:2.1}, and $\sigma(x)\in W_2^1(0, 1)$, $\sigma(0)+\sigma(1)=0$, and $\sigma(0)\neq-\frac{1}{2}$, in the case of Example \ref{exam:2.2}.
We recall that the conditions $\sigma(0)\neq-1$ and $\sigma(0)\neq-\frac{1}{2}$ provide the density of the domain $D(L_K)$ in $H$.

\section{Main results}
\label{sec:3}
In the Hilbert space $L_2(\Omega)$, where $\Omega$ is a bounded domain in $\mathbb R^m$ with an infinitely smooth boundary $\partial\Omega$, let us consider the minimal $L_0$ and maximal $\widehat{L}$ operators generated by the Laplace operator
\begin{equation}\label{eq:2.4}
-\Delta u=-\biggl(\frac{\partial^2 u}{\partial{x_1^2}}+\frac{\partial^2 u}{\partial{x_2^2}}+\cdots+\frac{\partial^2 u}{\partial{x_m^2}}\biggr).
\end{equation}

The closure $L_0$, in the space $L_2(\Omega)$ of the Laplace operator \eqref{eq:2.4} with the domain $C_0^\infty (\Omega)$, is \textit{the minimal operator corresponding to the Laplace operator}.
The operator $\widehat{L}$, adjoint to the minimal operator $L_0$ corresponding to the Laplace operator, is \textit{the maximal operator corresponding to the Laplace operator} (see \cite{Hormander}). Note that
 \[D(\widehat{L})=\{u\in L_2(\Omega): \; \widehat{L}u =-\Delta u \in L_2(\Omega)\}.\]
Denote by $L_D$ the operator, corresponding to the Dirichlet problem with the domain
\[ D(L_D)=\{u\in W_2^2(\Omega): \; u|_{\partial\Omega}=0\}. \]
Then, by virtue of \eqref{eq:1.1}, the inverse operators $L^{-1}$ to all possible correct restrictions of the maximal operator $\widehat{L}$ corresponding to the Laplace operator \eqref{eq:2.4} have the following form:
\[u\equiv L^{-1}f=L_D^{-1}f+Kf,\]
where, by virtue of \eqref{eq:1.1}, $K$ is an arbitrary linear operator bounded in $L_2(\Omega)$ with 
\[R(K)\subset \mbox{Ker}\,\widehat{L} =\{u\in L_2(\Omega): \: -\Delta u=0\}.\]
Then the direct operator $L$ is determined from the following problem:
\[\widehat{L}u= -\Delta u=f, \quad f\in L_2(\Omega),\]
\[D(L)=\{ u\in D(\widehat{L}) : \; [(I-K\widehat{L})u]|_{\partial\Omega}=0 \},\]
where $I$ is the identity operator in $L_2(\Omega)$. 
There are no other linear correct restrictions of the operator $\widehat{L}$ (see \cite{Biyarov1}).
The operators $(L^*)^{-1}$, corresponding to the adjoint operators $L^*$
\[v= (L^*)^{-1}g=L_D^{-1}g+K^*g,\]
describe the inverse operators to all possible correct extensions of $L_0$ if and only if $K$ satisfies the condition (see \cite{Biyarov1}):
\[\mbox{Ker}\,(I+K^*L^*)=\{0\}. \]
Note that the last condition is equivalent to the following: $\overline{D(L)}= L_2(\Omega)$.

We apply Theorem \ref{theorem:2.1} to the particular case when
\[Kf=\omega(x) \iint\limits_\Omega f(\xi)\overline{g(\xi)}d\xi,\quad x,\;\xi \in \Omega\subset \mathbb R^m,\]
where $\omega(x)$ is a harmonic function from $L_2(\Omega)$, and $g(x)\in L_2(\Omega)$.
\[K^*f=g(x) \iint\limits_\Omega f(\xi)\overline{\omega(\xi)}d\xi.\]
From the conditions of Theorem \ref{theorem:2.1} it follows that $g(x)\in W_2^2(\Omega)$, $g(x)\mid_{\partial \Omega}=0$, and
\[\iint \limits_\Omega(\Delta g)(\xi)\overline{\omega(\xi)}d\xi\neq1.\]
Then
\[B_Ku=-\Delta u-\omega(x)\iint \limits_\Omega(\Delta u)(\xi)(\Delta \overline{g})(\xi)d\xi=f(x),\quad \mbox{for all}\, \, f\in L_2(\Omega),\]
\[D(B_K)=\Bigl\{u\in W_2^2(\Omega):\, \bigl(u(x)+\omega(x)\iint \limits_\Omega(\Delta u)(\xi)\overline{g(\xi)} d\xi\bigr)\mid_{\partial\Omega}=0\Bigr\}\]
We obtained a relatively compact perturbation $B_K$ of $L$ which has the same eigenvalues as the Dirichlet problem $L_D$. The system of root vectors of $B_K$ form a Riesz basis in $L_2(\Omega)$. 
If $\{v_k\}$ are an orthonormal system of eigenfunctions of $L$ (the Dirichlet problem), then the system of eigenvectors $\{u_k\}$ of $B_K$ have the form
\[u_k=(I+\overline{KL})v_k=v_k(x)+\omega(x)\iint\limits_\Omega v_k(\xi)(\Delta\overline{g})(\xi)d\xi, \quad k=1, 2, \ldots\]

Consider a more visual case when $m=2$, that is, $\Omega\subset \mathbb R^2$. To do this, we set the operator $K$ using the functions $g(x)$ in the following form: let $z_1, z_2, \ldots, z_n=x_1^{(n)}+ix_2^{(n)}$ points lying strictly inside the domain $\Omega$. 
We take a holomorphic function $F(z)\in L_2(\Omega)$ in the domain $\Omega$ such that $F(z_k)=0,\,\, k=1, 2,\ldots, n$, with multiplicities $m_k$. As functions $g(x_1, x_2)$ we take the solution of the following Dirichlet problem
\begin{equation}\label{eq:2.5}
-\Delta g=\ln \vert F(z)\vert, \quad g\vert_{\partial\Omega}=0.
\end{equation}
Then, near the point where $F(z)\neq 0$ there is an  analytic branch $\Phi(z)$ of the function $\ln F$, hence $\ln|F|=\operatorname{Re} \Phi$  is a harmonic function. In a neighborhood of $z_k$ we can write
\[ 
\begin{split}
&F(z)=(z-z_k)^{m_k} \Phi(z), \\
&\ln |F(z)|=m_k\ln |z-z_k| +\ln |\Phi(z)|,
\end{split}
\]
where $\Phi(z_k)\neq0, \, k=\overline{1, n}$. Then by Theorem 3.3.2 (see \cite{Hormander})  and the harmonicity of the functions $\ln |\Phi(z)|$ we get that
\[\Delta\ln|F|=2\pi m_k\delta(z-z_k)\]
in the neighborhood.
We verify the condition of Theorem \ref{theorem:2.1}, taking into account that $\Omega\subset \mathbb R^2$ and
\[Kf=w(x)\iint\limits_\Omega f(\xi)\overline{g(\xi)} d\xi, \quad x, \xi \in \Omega\subset \mathbb R^2, \]
where $w(x)$ is a harmonic function from $L_2(\Omega)$ and $g(x)$ is a solution of the Dirichlet problem \eqref{eq:2.5}. Then $g(x)\in W_2^2(\Omega), \,\, g(x)|_{\partial\Omega}=0$, and
\[\iint \limits_\Omega \ln |F(\zeta)| \overline{w(\xi)} d\xi\neq1,\]
where $\zeta=\xi_1+i\xi_2$ and $\xi=(\xi_1, \xi_2)$.
If we denote by $T$ the next bounded operator in $L_2(\Omega)$
\[Tu=w(x)\int\limits_{\partial\Omega}\left[\frac{\partial u(\xi)}{\partial n}\ln |F(\zeta)|-u(\xi)\frac{\partial}{\partial n}\ln |F(\zeta)|\right] ds,\]
we get the following
\[B_K u=-\Delta u+2\pi w(x)\sum \limits_{k=1}^{n}m_k u(x^{(k)})-Tu=f(x),\]
where $x^{(k)}=(x^{(k)}_1, x^{(k)}_2)\in \Omega\subset\mathbb{R}^2$.
The domain of the operator $B_K$ has the form
\[ 
\begin{split}
D(B_K)=\bigg\{u\in &W_2^2(\Omega):\\
&\bigg[u(x)+w(x)\int\limits_{\partial\Omega}u(\xi)\frac{\partial\ln|F(\zeta)|}{\partial n}ds-w(x)\iint \limits_{\Omega}u(\xi)\ln|F(\zeta)| d\xi\bigg]\bigg|_{\partial\Omega}=0\bigg\}.
\end{split}
\]
We obtained a relatively bounded perturbation $B_K$ of $L_D$ which has the same eigenvalues as the Dirichlet problem $L_D$. The system of root vectors of  $B_K$ forms a Riesz basis in $L_2(\Omega)$. If $\{v_k\}$ are an orthonormal system of eigenfunctions of $L_D$, then the system of eigenvectors $\{u_k\}$ of $B_K$ have the form
\[u_k=(I+\overline{KL})v_k=v_k(x)+w(x)\iint\limits_\Omega v_k(\xi)\ln\overline{|F(\zeta)|}d\xi, \quad k=1, 2, \ldots.\]

Thus, we constructed an example of a singular perturbation of the Dirichlet problem for the Laplace operator with a basic system of root vectors. This perturbation is a valid non-self-adjoint operator, which is not a restriction of the maximal operator $\widehat{L}$ and is not an extension of the minimal operator $L_0$.

Using the properties of subharmonic functions, it is easy to obtain a similar result in the case of $n>2$.



\end{document}